\documentclass[a4paper,12pt,reqno]{amsart}

\usepackage{amsmath}
\usepackage{amssymb}
\usepackage{amsfonts}
\usepackage{graphicx}
\usepackage{color}
\usepackage[colorlinks]{hyperref}
\renewcommand\eqref[1]{(\ref{#1})} %Need with hyperref

\graphicspath{ {images/} }
\title[Stein inequalities for the Fourier transform]{Improved Stein inequalities \\ for the Fourier transform}

\author[E. D. Nursultanov]{Erlan D. Nursultanov}
\address{
	Erlan D. Nursultanov:
	\endgraf
          Department of Mathematics and Informatics
	\endgraf
	Lomonosov Moscow State University, Kazakhstan Branch
	\endgraf
	and
	\endgraf
	Institute of Mathematics and Mathematical Modeling, Almaty
		\endgraf
	Kazakhstan
	\endgraf
	{\it E-mail address} {\rm er-nurs@yandex.kz}}

\author[D. Suragan]{Durvudkhan Suragan}
\address{
	Durvudkhan Suragan:
	\endgraf
	Department of Mathematics
	\endgraf
	Nazarbayev University
	\endgraf
	Kazakhstan
	\endgraf
	{\it E-mail address} {\rm durvudkhan.suragan@nu.edu.kz}}
	

\thanks{This research is funded by the Committee of Science of the Ministry of Science and Higher Education of Kazakhstan (Grant No. AP19674900). This work was also supported by the Nazarbayev University grant 20122022FD4105.
No new data was collected or generated during the course of this research.}

\newtheoremstyle{theorem}%name
{10pt}          % space above
{10pt}  % space below
{\sl}  % bofy font
{\parindent}     % ident - empty=no indent,  \parindent= paragraph indent
{\bf}  % thm head font
{. }    % punctuation after thm head
{ }    % space after thm head: `` ``=normal \newline=linebreak
{}     % thm head specification
\theoremstyle{theorem}

\numberwithin{equation}{section}
\theoremstyle{plain}
\newtheorem{thm}{Theorem}[section]

\newtheorem{lem}[thm]{Lemma}

\theoremstyle{definition}

\newtheorem{rem}[thm]{Remark}

\newtheoremstyle{defi}%name
{10pt}          % space above
{10pt}  % space below
{\rm}  % bofy font
{\parindent}     % ident - empty=no indent,  \parindent= paragraph indent
{\bf}  % thm head font
{. }    % punctuation after thm head
{ }    % space after thm head: `` ``=normal \newline=linebreak
{}     % thm head specification
\theoremstyle{defi}

\begin{document}

 	\begin{abstract} In this paper, we present a refined version of the (classical) Stein inequality for the Fourier transform, elevating it to a new level of accuracy.
Furthermore, we establish extended analogues of a more precise version of the Stein inequality for the Fourier transform, broadening its applicability from the range $1< p<2$ to
	 $2\leq p<\infty$.
	\end{abstract}

\maketitle

\section{Introduction}
The Fourier transform  of  a function $f\in L_1(\mathbb R^n)$ is defined by
$$
\widehat{f}(y)=\int_{\Bbb R^n}f(x)e^{-i(y,x)}dx,\quad
y\in {\mathbb R^n},
$$
where $(y,x)=\sum_{i=1}^n y_ix_i$.
The following inequalities related to integral properties of functions and
their Fourier transforms are well known:

Let $1<p<2,$\;$p'=\frac{p}{p-1},$ and $0<q\le\infty,$ then
\begin{equation}\label{1}
\|\widehat{f}\|_{L_{p',q}(\mathbb R^n)}\leq c\|f\|_{L_{p,q}(\mathbb R^n)}
\end{equation}
and, in particular,
\begin{equation}\label{2}
\int_0^\infty t^{p-2}\left(\widehat f^*(t)\right)^p dt\leq c\|f\|_{L_{p}(\mathbb R^n)}^p,
\end{equation}
where $L_{p,q}({\mathbb R})$ is the classical Lorentz space. This inequality is called the
Stein inequality  \cite{Ste}. The Stein inequality, also known as the Hardy-Littlewood-Stein inequality in recognition of the contributions of Hardy and Littlewood \cite{HL}, or sometimes referred to as the Fourier inequality, has been extensively studied in the literature. Many authors have investigated this inequality in Lebesgue spaces with general weights, and some of the early results in this direction can be found in \cite{Mu} and references therein.

Notably, Benedetto and Heinig made significant contributions to the field through their papers \cite{BH1}, \cite{BHJ2} (with Johnson), and \cite{BH2}, where they also introduced the Lorentz space method to obtain Stein-type inequalities in weighted Lebesgue spaces. Sinnamon took a different approach and worked on these inequalities in Lorentz spaces in \cite{Si4} and \cite{Si5}, further advancing the understanding of the topic.

To the best of our knowledge, one of the pioneering works on Stein's inequality in Lorentz spaces was published by
Persson \cite{Per1}  (see also \cite{Per2}).
It is worth noting that for $p = 2$, the analogues of the Stein inequality have different forms
\cite{Boch, TM2, TM1}.  It should be emphasized that these works and contributions have greatly enriched the field of Stein inequalities and have paved the way for further investigations in this area. Further generalizations to Lorentz spaces with more general parameters have been obtained in \cite{KNP1} and \cite{KNP2}. For a comprehensive discussion on this topic, we recommend \cite{NT2}
and references therein.

In \cite{Nur1, Nur2}, analogues of the Stein inequalities in anisotropic Lorentz spaces
$L_{\bf p,q}(\Bbb R^n)$, ${\bf p}=(p_1,\ldots,p_n)$, ${\bf q}=(q_1,\ldots,q_n)$ (see also \cite {Blo})  were obtained.
Thus, in particular, for $1<p<2$ the following inequality holds:
\begin{equation}\label{e1}
\left(\int_0^\infty\ldots \int_0^\infty t_1^{p-2}t_2^{p-2}(\widehat f^{*_1\ldots*_n}(t_1,\ldots,t_n))^pdt_1\ldots dt_n\right)^{1/p}\leq c\|f\|_{ L_{p}(\Bbb
R^n)},
\end{equation}
 where $f^{*_1\ldots*_n}$  is the repeated non-increasing rearrangement of the function $f$.

 In the case $n\geq 2$, inequality \eqref{e1} is more precise than \eqref{2}, i.e. inequality \eqref{2} follows from \eqref{e1}, but the converse is not true.

This paper aims to obtain an improved version of inequality \eqref{1}, providing a sharper result.
We also extend the inequalities \eqref{1} and \eqref{e1} to cover the cases $2\leq p<\infty$ and $2\leq {\bf p}<\infty$, respectively.

Notably, our recent publication \cite{NS2020} showcases improved analogues of the Hardy-Littlewood-Stein inequalities for double trigonometric series. Furthermore, in \cite{NS2020}, we establish a novel unified version of the Hardy-Littlewood-Stein inequalities for the Fourier series in regular systems.

%Thus, some results of the present paper are essential extensions of our previous results from the Fourier series to the Fourier transform.

\section{Main results}

Let $\mu$ be the $n$-dimensional Lebesgue measure given in $\Bbb R^n$. Let $f$ be a measurable function defined on $\Bbb R^n$.

A function
$$
f^*(t) = \inf \left\{\sigma:\; \mu \{x\in \Omega:\; |f(x)| > \sigma \}\leq t\right\}
$$
is called the non-increasing rearrangement of the function $f$.

Let $0<
p<\infty$ and $0< q\leq \infty$. The Lorentz space $L_{p,q}(\Bbb R^n)$ is the set of all measurable functions $f$ for which
$$
\|f\|_{L_{p,q}}=\left(\int_0^\infty(t^{1/p}
f^*(t))^q \frac{dt}{t}\right)^{1/q}<\infty,\;0< q<\infty,
$$
and
$$
\|f\|_{L_{p,\infty}}=\sup_{t>0}t^{1/p}f^*(t)<\infty
$$
for $q=\infty$.

Along with the usual non-increasing rearrangement  $f^*$, we will use the repeated non-increasing rearrangement  $f^{*_1\ldots*_n}$ of the function $f$.

Let $f(x_1,\ldots, x_n)$ be a measurable function with respect to the Lebesgue measure $\mu$. For fixed $x_2,\ldots, x_n\in \Bbb R$, the function
$f^{*_1}(t_1,x_2, \ldots,x_n)$ is a nonincreasing rearrangement  of the function $\psi_1(x_1)=f(x_1,\ldots,x_n)$. If  $ x_3,\ldots, x_n$ are given, then
$f^{*_1*_2}(t_1, t_2,x_3,\ldots, x_n)$ is a nonincreasing rearrangement of $\psi_2(x_2)=f^{*_1}(t_1,x_2, \ldots,x_n)$. Continuing this process, in $n$ steps we
arrive at the function $f^{*_1\ldots*_n}(t_1, \ldots,t_n)$ which is called a repeated rearrangement  of the function $f(x_1,\ldots,x_n)$.

We also use the notations
$$
f^\star(2^m):=f^{*_1 ...*_n}(2^{m_1},\ldots,2^{m_n})
$$
for $m\in \Bbb Z^n$,
and
$$
D_k=\{m\in \Bbb Z^n:\; m_1+\ldots+m_n=k\}
$$
for $k\in \Bbb Z$.

\begin{thm}\label{T1}
Let $1<p<2$ and $0< q\leq\infty$.

 If  $\; f\in L_{p,q}(\Bbb R^n)$, then we have
\begin{equation}
\sum_{k\in \Bbb Z}2^{\frac {kq}{p'}}\left(\sum_{m\in D_k}\left({\widehat f\;}^\star(2^m)\right)^2\right)^{q/2}\leq c\|f\|^q_{L_{p,q}(\Bbb R^n)}.
\end{equation}
If
$$
\sum_{k\in \Bbb Z}2^{\frac {kq}{p}}\left(\sum_{m\in D_k}\left({ f\;}^\star(2^m)\right)^2\right)^{q/2}<\infty,
$$
 then $\widehat f\in  L_{p',q}(\Bbb R^n)$ and
\begin{equation}
\|\widehat f\|^q_{ L_{p',q}(\Bbb R^n)}\leq c\sum_{k\in \Bbb Z}2^{\frac {kq}{p}}\left(\sum_{m\in D_k}\left({ f\;}^\star(2^m)\right)^2\right)^{q/2}.
\end{equation}
\end{thm}

\begin{rem}
According to Lemma \ref{l-1}, each of inequalities \eqref{v6} and \eqref{v7} implies inequality  \eqref{1}.
Let us show that the converse argument is not true. Let $r$ be an arbitrary natural number. Define the set $G_r^*=\cup_{{}_{m_i\geq 0}^{m_1+\ldots+m_n= r}}Q_m $
with $Q_m=[0,2^{m_1}]\times\ldots\times [0,2^{m_n}]$.
Consider the functions  $g(x)=\chi_{G_r^*}(x)$ and $f(x)=\left(F^{-1}\chi_{G_r^*}\right)(x)$, where $F^{-1}$ is the inverse Fourier transform and
$\chi_{G_r^*}(x)$ is the characteristic function of the set $G_r^*$.

For $1<p<2$, inequalities \eqref{v6} and \eqref{v7} yield the estimates
$$
2^{\frac r{p'}} r^{\frac12} \leq c\|f\|_{L_{p,q}},
$$
and
$$
 \|\widehat g\|_{L_{p',q}}\leq c 2^{\frac r{p}} r^{\frac12}.
$$
Also, for $1<p<2$ inequality \eqref{1} leads to estimates of a different order:
$$
2^{\frac r{p'}} r^{\frac1{p'}} \leq \|f\|_{L_{p,q}},
$$
and
 $$
\|\widehat g\|_{L_{p',q}}\leq c 2^{\frac r{p}} r^{\frac1{p}},
$$
which are less accurate, of course.
 \end{rem}

Let $M$ be the set of all parallelepipeds of the form $Q=Q_1\times\ldots\times Q_n$, where $Q_i$ are segments in $\Bbb R$.

\begin{thm}\label{T2}
Let $1<p<r<\infty$ and $0< q\leq\infty$. If  $\; f\in L_{p,q}(\Bbb R^n)\cap L_{1}(\Bbb R^n)$, then we have
\begin{equation}\label{s1}
\sum_{k\in \Bbb Z}2^{\frac {kq}{p'}}\left(\sum_{m\in D_k}\left(\overline{\widehat f}(2^m;M)\right)^r\right)^{ q/r}\leq c\|f\|^q_{L_{p,q}(\Bbb R^n)},
\end{equation}
where

$$
\overline{ \widehat f}(2^m;M)=\sup_{{}_{Q\in M}^{|Q_j|\geq
2^{m_j}}}\frac1{|Q|}\left|\int_Q \widehat f( x)d{x}\right|.
$$

\end{thm}

\section{The spaces $\mathfrak L_{p,q}(\Lambda)$ }

Define a system of sets $\Lambda=\left\{\Lambda_r\right\}_{r\in\Bbb Z}$ by
$$
\Lambda_r=\cup_{m\in D_r}
\{t\in \Bbb R^n_+:  \; 2^{m_i}\leq t_i<2^{m_i+1} \;\;  i=1,\ldots,n\}.
$$
Given $k\in \Bbb Z$ the set
$$
G_k=\cup_{r=-\infty}^k \Lambda_r
$$
is called a stepped hyperbolic cross of order $k$ in $\Bbb R^n.$

For a measurable function $f$, we define a sequence $\{a_k(f,\Lambda)\}_{k\in\Bbb Z}$ by the formula
$$
a_k(f,\Lambda)=\left(\frac1{2^k}\int_{ \Lambda_k}\left(f^{*_1\ldots*_n}(t_1,\ldots,t_n)\right)^2dt_1\ldots dt_n\right)^{\frac12},\;\;\;k\in\Bbb Z.
$$
Let $0<p\leq \infty$ and $0<q\leq	\infty$. The space $ \mathfrak L_{p,q}(\Lambda)$  is defined by the quasinorms
$$
\|f\|_{\mathfrak L_{p,q}(\Lambda)}=\left(\sum_{k\in \Bbb Z}\left(2^{\frac{k}{p}}a_k(f,\Lambda)\right)^q\right)^{\frac1q},
$$
 for $0<q\leq\infty,\;\; 0<p<\infty,$ and
$$
\|f\|_{\mathfrak L_{p,\infty}(\Lambda)}=\sup_{m\in \Bbb Z}2^{\frac mp}a_m(f,\Lambda),
$$
for $q=\infty, \; 0<p\le\infty$.

\begin{lem}\label{l0}
 Let $f$ be a measurable function. Let $f^*$ be its nonincreasing rearrangement and $f^{*_1\ldots*_n}$ be its repeated nonincreasing
 rearrangement. Let $G_k=\cup_{r=-\infty}^k\Lambda_r$  be a stepped hyperbolic cross of order $k$ in $\Bbb R^n_+$.
 Then the following inequalities hold:
 \begin{equation}\label{12}
 \int_{0}^{2^{k}}(f^*(s))^2ds\leq \int_{ G_k} (f^{*_1\ldots*_n}(t_1, ,\ldots,t_n))^2 dt_1\ldots dt_n,
 \end{equation}
 and
 \begin{equation}\label{13}
 \int_{ \Lambda_{k}} (f^{*_1\ldots*_n}(t_1, ,\ldots,t_n))^2 dt_1\ldots dt_n\leq \int_{2^{k-1}}^\infty(f^*(s))^2ds.
 \end{equation}
 Here and after we assume that the existence of the right-hand side implies the existence of the left-hand side, ensuring that the inequality holds.
 \end{lem}

 \begin{proof}
Let $(t_1,\ldots,t_n)\notin G_k$. Then the monotonicity of the function $ f^{*_1\ldots*_n}$  in each variable implies that for any  $s\in
Q_t=(0,t_1]\times\ldots\times(0,t_n]$,
we have $$ f^{*_1\ldots*_n}(t_1,\ldots,t_n)\leq  f^{*_1\ldots*_n}(s_1,\ldots,s_n).$$
Taking into account the fact $\; \mu Q_t=\prod_{i=1}t_i>2^k$, we get
 \begin{equation}\label{14}
 f^{*_1\ldots*_n}(t_1,\ldots,t_n)\leq f^*(2^k).
  \end{equation}
  Note that there is a measurable set $D\subset \Bbb R^n$ such that $\mu D=2^k$ and
  $$
\int_{0}^{2^{k}}(f^*(s))^2ds=\int_D|f(x)|^2dx.
$$
Let $D_1=\{x\in D: \; |f(x)|>f^*(2^k)\}$, and $D_2=\{x\in D: \; |f(x)|=f^*(2^k)\}$, then $\mu D_1+\mu D_2=\mu D=2^k$.

Notice that the nondecreasing rearrangement of the characteristic function  $\chi_{D},\; D\subset \mathbb{R}^{n},$ is the characteristic function of
some set $\tilde{D}\subset\mathbb{R}^{n}_{+}$.

By using Hardy's inequality for the rearrangements, we obtain
$$
\int_{0}^{2^{k}}(f^*(s))^2ds=\int_{\Bbb R^n}|f(x)|^2\chi_D(x)dx
$$
$$
\leq\int_0^\infty\ldots\int_0^\infty  \left(f^{*_1\ldots*_n}(t_1,\ldots,t_n)\right)^2\chi_{D}^{*_1\ldots*_n}(t_1,\ldots,t_n)dt_1\ldots dt_n
$$
$$
=\int_{\tilde D}\left(f^{*_1\ldots*_n}(t)\right)^2 dt=\int_{\tilde D_1}\left(f^{*_1\ldots*_n}(t)\right)^2 dt+\int_{\tilde D_2}\left(f^{*_1\ldots*_n}(t)\right)^2
dt
$$
$$
=\int_{\tilde D_1}\left(f^{*_1\ldots*_n}(t)\right)^2 dt+\left(f^*(2^k)\right)^2\mu  \tilde D_2,
$$
where $\tilde D_1=\{t\in \tilde D: \; f^{*_1\ldots*_n}(t_1,\ldots,t_n)>f^*(2^k)\},\;\;\tilde  D_2=\{x\in \tilde D: \;
f^{*_1\ldots*_n}(t_1,\ldots,t_n)=f^*(2^k)\}$, and $\mu \tilde D_i=\mu D_i, \;\;i=1,2$.
 Here $\mu \tilde{D}=\mu \tilde{D}_1+\mu \tilde{D}_2=2^k$.

It follows from relation \eqref{14} that $\tilde D_1\subset G_k$ . If $\tilde D_2\subset G_k$, then the first statement of the lemma is proved.

Let $\xi\in \tilde D_2\setminus G_k$, then $f^{*_1\ldots*_n}(\xi_1,\ldots,\xi_n)=f^*(2^k)$, and hence, from the monotonicity of $f^{*_1\ldots*_n}$ we obtain
$$
f^{*_1\ldots*_n}(s_1,\ldots,s_n)\geq f^*(2^k), \quad \forall s\in G_k\cap Q_\xi.
$$
Taking into account the facts $\mu\left(G_k\cap Q_\xi\right)>2^k$ and $\mu\left(\tilde D_1\cup\tilde D_2\right)=2^k$, we have
$$
\mu\left(\left(G_k\cap Q_\xi\right)\setminus \tilde D_1\right)\geq \mu \tilde D_2.
$$
Therefore, since $\tilde D_1\subset G_k$, we obtain
$$
\int_{ G_k} (f^{*_1\ldots *_n}(t))^2 dt=\int_{ \tilde D_1} (f^{*_1\ldots *_n}(t))^2 dt+\int_{ G_k\setminus \tilde D_1} (f^{*_1\ldots *_n}(t))^2 dt
$$
$$
\geq \int_{ \tilde D_1} (f^{*_1\ldots*_n}(t))^2 dt+\int_{\left( G_k\cap Q_\xi\right)\setminus \tilde D_1} (f^{*_1\ldots*_n}(t))^2 dt
$$
$$
\geq \int_{ \tilde D_1} (f^{*_1\ldots*_n}(t))^2 dt+( f^*(2^k))^2\mu\left( \left(G_k\cap Q_\xi\right)\setminus \tilde D_1\right)
$$
$$
\geq \int_{ \tilde D_1} (f^{*_1\ldots*_n}(t))^2 dt+( f^*(2^k))^2\mu \tilde D_2
\geq \int_{0}^{2^{k}}(f^*(s))^2ds.
$$
That is, inequality \eqref{12} is proved. The second relation follows from the first relation and
the equality
$$
\int_{0}^\infty(f^*(s))^2ds= \int_0^\infty\ldots\int_0^\infty  \left(f^{*_1\ldots*_n}(t_1,\ldots,t_n)\right)^2 dt_1\ldots dt_n.
$$
\end{proof}

\begin{lem}\label{l1}
 If $0<q<q_1\leq\infty,\;1<p\leq\infty, $ then
$$
\mathfrak L_{p,q}(\Lambda)\hookrightarrow \mathfrak L_{p,q_1}(\Lambda).
$$
\end{lem}
\begin{proof}
Let $f\in \mathfrak L_{p,q}(\Lambda)$. By using Jensen's inequality, we have
$$
\|f\|_{\mathfrak L_{p,q_1}(\Lambda)}=\left(\sum_{k\in \Bbb Z}\left(2^{\frac{k}{p}}a_k(f,\Lambda)\right)^{q_1}\right)^{\frac1{q_1}}\leq
\left(\sum_{k\in \Bbb Z}\left(2^{\frac{k}{p}}a_k(f,\Lambda)\right)^{q}\right)^{\frac1{q}}=\|f\|_{\mathfrak L_{p,q}(\Lambda)}<\infty,
$$
which implies $\mathfrak L_{p,q}(\Lambda)\hookrightarrow \mathfrak L_{p,q_1}(\Lambda)$.
\end{proof}
We will need Hardy's inequality in the following form.

 \begin{lem}\label{l-2}[Hardy's inequality]
 Let $\alpha>0$, $0<q, h\leq \infty$.
Then the following inequalities hold
 $$
 \left(\sum_{k=-\infty}^\infty\left(2^{-\alpha k }\left(\sum_{r=-\infty}^k|b_r|^h\right)^{\frac1h}\right)^q\right)^{\frac1q}\leq c_{\alpha,q,
 h}\left(\sum_{k=-\infty}^\infty\left(2^{-\alpha k }|b_k|\right)^q\right)^{\frac1q}
 $$
 and
 $$
 \left(\sum_{k=-\infty}^\infty\left(2^{\alpha k }\left(\sum_{r=k-1}^\infty|b_r|^h\right)^{\frac1h}\right)^q\right)^{\frac1q}\leq c_{\alpha,q,
 h}\left(\sum_{k=-\infty}^\infty\left(2^{\alpha k }|b_k|\right)^q\right)^{\frac1q}.
 $$
 \end{lem}

\begin{proof}
Let $0<h \leq q \leq \infty, 0<\varepsilon<\alpha$. We will use Hölder's inequality

$$\left(\sum_{k=-\infty}^{\infty}\left(2^{-\alpha k}\left(\sum_{r=-\infty}^k\left|b_r\right|^h\right)^{\frac{1}{h}}\right)^q\right)^{\frac{1}{q}} \leq$$$$\left(\sum_{k=-\infty}^{\infty}\left(2^{-\alpha k}\left(\sum_{r=-\infty}^k\left(2^{-\varepsilon r}\left|b_r\right|\right)^q\right)^{\frac{1}{q}}\left(\sum_{r=-\infty}^k 2^{\varepsilon \tau r}\right)^{\frac{1}{\tau}}\right)^q\right)^{\frac{1}{q}},$$

where $\frac{1}{\tau}=\frac{1}{h}-\frac{1}{q}$. Hence, we have

$$
\left(\sum_{k=-\infty}^{\infty}\left(2^{-\alpha k}\left(\sum_{r=-\infty}^k\left|b_r\right|^h\right)^{\frac{1}{h}}\right)^q\right)^{\frac{1}{q}} 
$$
$$
\lesssim\left(\sum_{k=-\infty}^{\infty} 2^{(\varepsilon -\alpha) q k} \sum_{r=-\infty}^k\left(2^{-\varepsilon r}\left|b_r\right|\right)^q\right)^{\frac{1}{q}}
$$

$$
=\left(\sum_{r=-\infty}^{\infty}\left({2 }^{-\varepsilon r}\left|b_r\right|\right)^q \sum_{k=r}^{\infty} 2^{(\varepsilon -\alpha) q k}\right)^{\frac{1}{q}} \lesssim\left(\sum_{r=-\infty}^{\infty}\left(2^{-\alpha r}\left|b_r\right|\right)^q\right)^{\frac{1}{q}}.$$

Similarly, we get

$$
\left(\sum_{k=-\infty}^{\infty}\left(2^{\alpha k}\left(\sum_{r=k}^{\infty}\left|b_r\right|^h\right)^{\frac{1}{h}}\right)^q\right)^{\frac{1}{q}}
 $$
$$
\leq
\left(\sum_{k=-\infty}^{\infty}\left(2^{\alpha k}\left(\sum_{r=k}^{\infty}\left(2^{\varepsilon r}\left|b_r\right|\right)^q\right)^{\frac{1}{q}}\left(\sum_{r=k}^{\infty} 2^{-\varepsilon \tau r}\right)^{\frac{1}{\tau}}\right)^q\right)^{\frac{1}{q}}$$

$$
\lesssim\left(\sum_{k=-\infty}^{\infty} 2^{(-\varepsilon +\alpha) q k} \sum_{r=k}^{\infty}\left(2^{\varepsilon r}\left|b_r\right|\right)^q\right)$$

$$=\left(\sum_{r=-\infty}^{\infty}\left(2^{\varepsilon r}\left|b_r\right|\right)^q \sum_{k=-\infty}^r 2^{(-\varepsilon+\alpha) q r}\right)^{\frac{1}{q}} \lesssim\left(\sum_{r=-\infty}^{\infty}\left(2^{\alpha r}\left|b_r\right|\right)^q\right)^{\frac{1}{q}}.$$

Let $0<q<h \leq \infty$. By using Jensen's inequality, we arrive at

$$\left(\sum_{k=-\infty}^{\infty}\left(2^{-\alpha k}\left(\sum_{r=-\infty}^k\left|b_r\right|^h\right)^{\frac{1}{h}}\right)^q\right)^{\frac{1}{q}} \leq\left(\sum_{k=-\infty}^{\infty} 2^{-\alpha q k} \sum_{r=-\infty}^k\left|b_r\right|^q\right)^{\frac{1}{q}}$$

$$=\left(\sum_{r=-\infty}^{\infty}\left|b_r\right|^q \sum_{k=r}^{\infty} 2^{-\alpha k q}\right)^{\frac{1}{q}} \leq\left(\sum_{r=-\infty}^{\infty}\left(2^{-\alpha r}\left|b_r\right|\right)^q\right)^{\frac{1}{q}}.$$

The second inequality also follows from Jensen's inequality.

\end{proof}

\begin{lem}\label{l-1}
For  $0<q\leq\infty,$ we have the embeddings
\begin{equation}\label{n1}
L_{p,q}(\Bbb R^n)\hookrightarrow \mathfrak L_{p,q}\left(\Lambda\right),\;1<p<2,
\end{equation}
and
\begin{equation}\label{n2}
\mathfrak L_{p,q}\left(\Lambda\right)\hookrightarrow L_{p,q}(\Bbb R^n),\;2<p<\infty.
\end{equation}
\end{lem}
\begin{proof}
Let $1<p<2, \; f\in L_{p,q}$. By using relation \eqref{13} from Lemma \ref{l0}, we have
$$
\|f\|_{\mathfrak L_{p,q}\left(\Lambda\right)}=\left(\sum_{r=-\infty}^\infty \left(2^{\frac rp}\left(\frac1{2^r}\int_{\Lambda_r}(f^{*_1\ldots*_n}(t))^2\;
dt\right)^{\frac12}\right)^q\right)^{\frac1q}
$$
$$
\leq\left(\sum_{r=-\infty}^\infty \left(2^{\frac rp}\left(\frac1{2^r}\int_{2^{r-1}}^\infty( f^*(\xi))^2d\xi\right)^{\frac12} \right)^q\right)
^{\frac1q}
$$
$$
=\left(\sum_{r=-\infty}^\infty \left(2^{\frac rp}\left(\frac1{2^r}\sum_{k=r-1}^\infty\int_{2^k}^{2^{k+1}}( f^*(\xi))^2d\xi\right)^{\frac12}
\right)^q\right)^{\frac1q}.
$$
Further, by applying Hardy's inequality (Lemma \ref{l-2}), we obtain
$$
\|f\|_{\mathfrak L_{p,q}\left(\Lambda\right)}\lesssim
\left(\sum_{r=-\infty}^\infty \left(2^{\frac rp}\left(\frac1{2^r}\int_{2^r}^{2^{r+1}}( f^*(\xi))^2d\xi\right)^{\frac12} \right)^q\right)^{\frac1q}
$$
$$
\asymp \left(\sum_{r=-\infty}^\infty \left(2^{\frac rp} f^*(2^r)\right)^q\right)^{\frac1q}\asymp\|f\|_{L_{p,q}(\Bbb R^n)}.
$$
Thus, we establish \eqref{n1}.

Let $f\in \mathfrak L_{p,q}(\Lambda),\; 2<p<\infty$. By using  \eqref{12}, we compute
$$
\|f\|_{L_{p,q}(\Bbb R^n)}\leq \left(\sum_{r=-\infty}^\infty \left(2^{\frac rp}\left(\frac1{2^r}\int_{0}^{2^r} f^*(\xi)d\xi\right) \right)^q\right)^{\frac1q}
$$
$$
\leq\left(\sum_{r=-\infty}^\infty \left(2^{\frac rp}\left(\frac1{2^r} \int_{0}^{2^r}( f^*(\xi))^2d\xi\right)^{\frac12} \right)^q\right)^{\frac1q}
$$
$$
\leq
\left(\sum_{r=-\infty}^\infty \left(2^{\frac rp}\left(\frac1{2^r}\int_{ G_r}(f^{*_1\ldots*_n}(t))^2dt\right)^{\frac12}\right)^q\right)^{\frac1q}
$$
$$
=\left(\sum_{r=-\infty}^\infty \left(2^{\frac rp}\left(\frac1{2^r}\sum_{k=-\infty}^r\int_{
\Lambda_k}(f^{*_1\ldots*_n}(t))^2dt\right)^{\frac12}\right)^q\right)^{\frac1q}
\lesssim \|f\|_{\mathfrak L_{p,q}(\Lambda)}.
$$
Lemma \ref{l-2} has been used in the last line.
\end{proof}
The following statement is a consequence of  \cite[Lemma 1]{Nur00}.
 \begin{lem}\label{l7}
If  $1<p<2, \; p'=\frac{p}{p-1}$, then we have
 $$
\left(\int_0^\infty\ldots\int_0^\infty \left((t_1\ldots t_n)^{\frac 1 {p'}} g^{*_1\ldots*_n}(t_1,\ldots,t_n)\right)^2
\frac{dt_1}{t_1}\ldots\frac{dt_n}{t_n}\right)^{1/2}
$$
$$
\leq c_p\left(\int_0^\infty\ldots\int_0^\infty (t_1\ldots t_n)^{p-2}\left(g^{*_1\ldots*_n}(t_1,\ldots,t_n)\right)^p dt_1\ldots dt_n\right)^{1/p}.
$$
  \end{lem}

  \begin{lem}\label{l8}
If  $0<p<\infty, \; 0<q\leq\infty$, then we have
 $$
\|f\|^q_{\mathfrak L_{p,q}\left(\Lambda\right)}\asymp\sum_{k\in \Bbb Z}2^{\frac {kq}{p}}\left(\sum_{m\in D_k}\left(f^\star(2^m)\right)^2\right)^{q/2}.
$$
  \end{lem}
The proof of Lemma \ref{l8} can be carried out by the standard method.

  \medskip

  Let $(A_0,A_1)$ be a compatible pair of Banach spaces \cite{BL}. Recall that
$$
 K(t,a;A_0,A_1)
= \inf_{a=a_0+a_1} (\|a_0\|_{A_0} + t\|a_1\|_{A_1}),\; a \in A_0 +A_1, \; t>0,
$$
is the Peetre functional.

For $0<\theta<1$, we have
$$
(A_0,A_1)_{\theta,q} = \left\{ a \in A_0 + A_1 :
\|a\|_{(A_0,A_1)_{\theta ,q}}
= \left (\int\limits_{0}^{\infty} (t^{-\theta}K(t,a))^q\frac  {{dt}}{{t}}
\right )^{\frac {1}{q}} < \infty  \right\} ,
$$
 for $0< q < \infty$  and
$$(A_0,A_1)_{\theta,\infty} = \left\{ a \in A_0 + A_1:
\|a\|_{(A_0,A_1)_{\theta ,\infty}}
= \sup_{0<t<\infty} t^{-\theta}K(t,a) < \infty \right\},
$$
for $q=\infty$.

The following interpolation theorem holds for the spaces $\mathfrak L_{p,q}(\Lambda)$.
\begin{thm}\label{T4}
Let  $1\leq p_0<p_1\leq\infty,$ $1\leq q_0,q_1,q\leq\infty,$ and $1<\theta<\infty.$ Then
\begin{equation}\label{10}
\left(\mathfrak L_{p_0,q_0}(\Lambda),\;\mathfrak L_{p_1,q_1}(\Lambda)\right)_{\theta, q}\hookrightarrow \mathfrak L_{p,q}(\Lambda),
\end{equation}
 where
 $1/p=(1-\theta)/p_0+\theta/p_1$.
 \end{thm}
\begin{proof}
 Let $f=f_0+f_1$ be an arbitrary representation, where $f_0\in \mathfrak L_{p_0,\infty}(\Lambda),$ $f_1\in \mathfrak L_{p_1,\infty}(\Lambda)$ and
 $f\in\left(\mathfrak L_{p_0,q_0}(\Lambda),\;\;\mathfrak L_{p_1,q_1}(\Lambda)\right)_{\theta, q}$. For $n\in {\mathbb N}$, taking into account
$$
(f_0+f_1)^{*_1\ldots*_n}(t_1,\ldots,t_n)\leq f_0^{*_1\ldots*_n}\left(\frac{t_1}2,\ldots,\frac{t_n}2\right)+f_1^{*_1\ldots*_n}
\left(\frac{t_1}2,\ldots,\frac{t_n}2\right),
$$
we obtain
$$
\left(\frac{1}{2^r}\int_{\Lambda_r}\left(f^{*_1\ldots*_n}(t)\right)^2dt\right)^{\frac12}\leq2^{n/2}
\left(\left(\frac1{2^{r}}\int_{\Lambda_{r-n}}\left(f_0^{*_1\ldots*_n}(t)\right)^2dt\right)^{\frac12}\right.
$$$$
\left. +\left(\frac1{2^{r}}\int_{\Lambda_{r-n}}\left(f_1^{*_1\ldots*_n}(t)\right)^2dt\right)^{\frac12}\right)
$$
$$
\lesssim  2^{-\frac{r}{p_0}}\left(\|f_0\|_{\mathfrak L_{p_0,\infty}(\Lambda)}+2^{ r\left(\frac1{p_0}-\frac1{p_1}\right)}\|f_1\|_{\mathfrak
L_{p_1,\infty}(\Lambda)}\right).
$$
The arbitrariness of the representation $f=f_0+f_1$ implies 
$$
\left(\frac{1}{2^r}\int_{\Lambda_r}\left(f^{*_1\ldots*_n}(t)\right)^2dt\right)^{\frac12}
\lesssim 2^{-\frac r{p_0}}K\left(2^{(\frac1{p_0}-\frac1{p_1})r},
f; \mathfrak
L_{p_0,\infty}(\Lambda),\mathfrak
L_{p_1,\infty}(\Lambda)\right).
$$
Therefore, we have
$$
\|f\|_{\mathfrak L_{p,q}(\Lambda)}\lesssim\left(\sum_{r=-\infty}^\infty \left(2^{(\frac 1p-\frac 1{p_0})r}K\left(2^{(\frac1{p_0}-\frac1{p_1})r},
f; \mathfrak
L_{p_0,\infty}(\Lambda),\mathfrak
L_{p_1,\infty}(\Lambda)\right)\right)^q\right)^{\frac1q}
$$
$$
=\left(\sum_{r=-\infty}^\infty \left(2^{-\theta r(\frac 1{p_0}-\frac 1{p_1})}K\left(2^{(\frac1{p_0}-\frac1{p_1})r},
f; \mathfrak
L_{p_0,\infty}(\Lambda),\mathfrak
L_{p_1,\infty}(\Lambda)\right)\right)^q\right)^{\frac1q}
$$

$$
=\left(\sum_{r=-\infty}^\infty \left(b^{-\theta r}K\left(b^r, f; \mathfrak
L_{p_0,\infty}(\Lambda),\mathfrak
L_{p_1,\infty}(\Lambda)\right)\right)^q\right)^{\frac1q}
$$
$$
\asymp\|f\|_{\left(\mathfrak L_{p_0,\infty}(\Lambda),\;\mathfrak L_{p_1,\infty}(\Lambda)\right)_{\theta, q}},
$$

with $b=2^{(\frac 1{p_0}-\frac 1{p_1})}$.

Now combining it with the embedding $\mathfrak L_{p_i,q_i}(\Lambda)\hookrightarrow \mathfrak L_{p_i,\infty}(\Lambda), \;i=0,1$, we arrive at \eqref{10}.
\end{proof}

\begin{rem}
In the case $n=1$, $\mathfrak L_{p,q}(\Lambda)$ coincides with the Lorentz space
 $L_{p,q}(\Bbb R)$, that is, the following equality holds
$$
\left(\mathfrak L_{p_0,q_0}(\Lambda),\;\mathfrak L_{p_1,q_1}(\Lambda)\right)_{\theta, q}= \mathfrak L_{p,q}(\Lambda).
$$
The question of whether this equality holds for $n>1$ remains open.
\end{rem}

\section{Proof of Theorem \ref{T1}}

Based on Lemma \ref{l8}, it is only necessary to establish the following inequalities for the case where $1<p<2$,

\begin{equation}\label{v6}
\|\widehat f\|_{\mathfrak L_{p',q}(\Lambda)}\leq c\|f\|_{L_{p,q}(\Bbb R^n)},
\end{equation}
 and
\begin{equation}\label{v7}
\|\widehat f\|_{ L_{p',q}(\Bbb R^n)}\leq c\|f\|_{\mathfrak L_{p,q}(\Lambda)}.
\end{equation}

Let us first prove the weak inequality
\begin{equation}\label{v8}
\|\widehat f\|_{\mathfrak L_{p',\infty}(\Lambda)}\leq c_p\|f\|_{L_p(\Bbb R^n)}.
\end{equation}
Let $f\in L_p(\Bbb R^n)$ and $G_r$ with $r\in  \Bbb Z$ be the stepped hyperbolic cross in $\Bbb R^n_+$. We have
$$
2^{\frac r{p'}}a_r(\widehat f, \Lambda)=2^{\frac r{p'}}\left(\frac1{2^n}\int_{ \Lambda_r}\left(\widehat f^{*_1\ldots*_n}(t)\right)^2dt\right)^{\frac12}
$$
$$
\leq\left(\int_{ G_n}\left((t_1\ldots t_n)^{\frac 1 {p'}}\widehat f^{*_1\ldots*_n}(t_1,\ldots,t_n)\right)^2\frac{dt_1}{t_1}\ldots\frac{dt_n}{t_n}\right)^{1/2}
$$
$$
\leq\left(\int_0^\infty\ldots\int_0^\infty \left((t_1\ldots t_n)^{\frac 1 {p'}}\widehat f^{*_1\ldots*_n}(t_1,\ldots,t_n)\right)^2
\frac{dt_1}{t_1}\ldots\frac{dt_n}{t_n}\right)^{1/2}.
$$

By using Lemma \ref{l7} and inequality \eqref{e1}, we obtain

$$
2^{\frac r{p'}}a_r(\widehat f, \Lambda)\lesssim \left(\int_0^\infty\ldots\int_0^\infty (t_1\ldots t_n)^{p-2}\left(\widehat
f^{*_1\ldots*_n}(t_1,\ldots,t_n)\right)^p dt_1\ldots dt_n\right)^{1/p}
 \lesssim\| f\|_{L_{p}(\Bbb R^n)}.
$$
Thus, taking into account the arbitrariness of the parameter $r$, we get
$$
\|\widehat f\|_{\mathfrak L_{p',\infty}(\Lambda)}\leq c\|f\|_{L_p(\Bbb R^n)}.
$$
Further, let $1<p<2$, then there are $p_0, p_1, \theta$ such that

$$
1<p_0<p< p_1<2, \quad \frac1p=\frac{1-\theta}{p_0}+\frac\theta{p_1}.
$$
From \eqref{v8} we establish
$$
\|\widehat f\|_{\mathfrak L_{p_0',\infty}(\Lambda)}\leq c_{p_0}\|f\|_{L_{p_0}(\Bbb R^n)},
$$
and
$$
\|\widehat f\|_{\mathfrak L_{p_1',\infty}(\Lambda)}\leq c_{p_1}\|f\|_{L_{p_1}(\Bbb R^n)}.
$$
Thus, according to the real interpolation method, we obtain
$$
\|\widehat f\|_{\left(\mathfrak L_{p_0',\infty}(\Lambda), \mathfrak L_{p_1',\infty}(\Lambda)\right)_{\theta,q}}\leq
c(c_{p_0})^{1-\theta}(c_{p_1})^\theta\|f\|_{\left(L_{p_0}(\Bbb R^n),L_{p_1}(\Bbb R^n)\right)_{\theta,q}}.
$$
Taking into account Theorem \ref{T4} and the fact that
$\left(L_{p_0}(\Bbb R^n),L_{p_1}(\Bbb R^n)\right)_{\theta,q}=L_{p,q}(\Bbb R^n)$
(see, e.g. \cite{BL}) we get \eqref{v6}.

Now let us show the second statement in Theorem \ref{T1}. Let $f\in \mathfrak L_{p,q}(\Lambda)$.
From the dual representation of the norm of the Lorentz space and the Plancherel theorem, we have
$$
\|\widehat f\|_{L_{p',q'}}\asymp \sup_{\|g\|_{L_{p,q}}=1}\int_{\Bbb R^n}\widehat f(x)\overline{g(x)}dx = \sup_{\|g\|_{L_{p,q}}=1}\int_{\Bbb R^n}
f(x)\overline{\widehat g(x)}dx
$$
$$
\leq\sup_{\|g\|_{L_{p,q}}=1}\int_0^\infty\ldots\int_0^\infty f^{*_1\ldots*_n}(t_1,\ldots,t_n) \widehat g^{*_1\ldots*_n}(t_1,\ldots,t_n)dt_1\ldots dt_n
$$
$$
=\sup_{\|g\|_{L_{p,q}}=1}\sum_{r=-\infty}^\infty\int_{\Lambda_r}f^{*_1\ldots*_n}(t_1,\ldots,t_n) \widehat g^{*_1\ldots*_n}(t_1,\ldots,t_n)dt_1\ldots dt_n
$$
$$
\leq\sup_{\|g\|_{L_{p,q}}=1}\sum_{r=-\infty}^\infty\left(\int_{\Lambda_r}(f^{*_1\ldots*_n}(t))^2 dt\right)^{\frac12}\left(\int_{\Lambda_r}(\widehat
g^{*_1\ldots*_n}(t))^2 dt\right)^{\frac12}
$$
$$
\leq\sup_{\|g\|_{L_{p,q}}=1}\|f\|_{\mathfrak L_{p,q'}(\Lambda)}\|\widehat g\|_{\mathfrak L_{p',q}(\Lambda)}.
$$
The condition $1<p<2$  implies $2<p'<\infty.$
This allows us to apply inequality \eqref{v6}.
Thus, we arrive at \eqref{v7}.

 \section{Fourier transform in anisotropic Lorentz spaces}

 Let ${\bf p}=(p_1,\ldots,p_n),\;{\bf
q}=(q_1,\ldots,q_n)$ be vectors such that if $0<q_j<\infty$, then $0<p_j<\infty$, and if $q_j=\infty$, then $0<p_j\le\infty$, $j=1,\ldots,n.$
%We will assume that the vectors $\bf p$  and $\bf q$ satisfy these conditions.

Define the functional
$$
\Phi_{\bf p,q}(\varphi)=\left(\int_0^\infty\ldots\left(\int_0^\infty
         \left|t_1^\frac1{p_1}\ldots
t_n^\frac1{p_n}\varphi(t_1,\ldots,t_n)\right|^{q_1}\frac{dt_1}{t_1}\right)^\frac{q_2}{q_1}\ldots
\frac{dt_n}{t_n}\right)^\frac1{q_n},
$$
here for $q=\infty$  the expression $\left(\int_0^\infty (G(t))^q\frac{dt}t\right)^{1/q}$ is understood as $\sup_{t>0}G(t)$.

The space $L_{\bf p, q}^*({{\mathbb R}}^{n})$ is defined as the set of functions such that
$$ \| f \|_{L_{\bf
p,q}^*({{\mathbb R}}^{n})} = \Phi_{\bf p,q}
(f^{*_1\ldots*_n})<\infty .$$

Note that these spaces were studied by Blozinsky \cite{Blo}.

Let $M$ be the set of all parallelepipeds of the form $Q=Q_1\times\ldots\times Q_n$, where $Q_i$ are segments in $\Bbb R$.

For an integrable function $f(x)=f(x_1,\ldots,x_n)$  (on each $Q$)  which is contained in $M$, we define the function
 $$
{\bar f}(t;M)={\bar f}(t_1,\ldots,t_n;M)=\sup_{{}_{Q\in M}^{|Q_j|>
t_j}}\frac1{|Q|}\left|\int_Q f( x)d{x}\right|,
$$
where $|Q_j|$ is the length of the segment $Q_j$.

Similarly, we define the space
$$
N_{\bf p,q}(M)=\left\{
f :\;\; \|f\|_{N_{\bf p,q}(M)}=\Phi_{\bf
p,q}\left(\{\bar f(\cdot, M)\}\right)<\infty\right\}.
$$
For ${\bf q} \leq {\bf q_1}\;\;(q_j \leq q_j^1,\;\; j=1,\ldots,n)$, there are embeddings
$$
L_{\bf p, q}(\Omega)\hookrightarrow L_{\bf p,
q_1}(\Omega),\;\;\;\;\;\;N_{\bf p, q}(M)\hookrightarrow N_{\bf p, q_1}(M).
$$

Let  ${\bf A_{\bf p,\sigma}}=\left(L_{p_{1},\sigma_1}(\mathbb R),
\ldots,L_{p_{n},\sigma_n}(\mathbb R)\right)$  be the Lorentz space with mixed metric, which is determined by the norm
$$
\|f\|_{{\bf A_{\bf p,\sigma}}}=\|\ldots\|f\|_{L_{p_{1},\sigma_1}(\mathbb R)}\ldots\|_{L_{p_{n},\sigma_n}(\mathbb R)}.
$$
This space does not coincide with the anisotropic Lorentz space $L_{\bf p,q}(\Bbb R^n)$  and
is a generalization of the Lebesgue space with the mixed norm $L_{\bf p}(\Bbb R^n)$. In turn, we need an interpolation theorem, which is a consequence of Theorem
1 and Theorem 2 from \cite{Nur00}.

Let $E=\{\varepsilon=(\varepsilon_1,\ldots,\varepsilon_n):\; \varepsilon_i\in \{0,1\}\}$ be the vertices of the unit cube.
\begin{thm} \label{T4-2}
Let ${\bf0}< {\bf p_0,\; p_1}<{\bf\infty}$, $ p_i^{0}\neq p_i^1,\; i=1,\ldots,n,\;
{{\bf 0}<\theta=(\theta_1,\ldots,\theta_n)<1}$,
$1/{\bf p} = {1-\theta }/{\bf p_0} + \theta/{\bf p_1},\;\;
0<{\bf q}\leq \infty.$

If $T$ is a linear operator such that $
T : \bf {A_{\bf p_\varepsilon,1}} \to N_{\bf p'_\varepsilon,\infty}$, where ${\bf p}_\varepsilon=(p_1^{\varepsilon_1},\ldots,p_n^{\varepsilon_n})$, with the norm $
M_\varepsilon$ for any $\varepsilon\in E$, then we have
$$
T : L_{\bf p,q}(\Bbb R^n)\to N_{\bf p',q}(M),
$$
with the norm $\|T\|\le\max_{\varepsilon\in E}M_\varepsilon$.
\end{thm}

\begin{thm} \label{T3}
Let $ {\bf1}< {\bf p}<{\bf\infty}$, $ {\bf 0}< {\bf q}\leq {\bf\infty}$. If $f\in L_{\bf p,q}(\Bbb R^n)$, then $\widehat f\in N_{\bf p',q}(M)$ and
\begin{equation}\label{v3}
\|\widehat f\|_{N_{\bf p',q}(M)} \leq c \|f\|_{L_{\bf p,q}(\Bbb R^n)}.
\end{equation}
\end{thm}

\begin{proof}

%Let us proceed directly to the proof of Theorem \ref{T3}.

Let $1<{\bf p}<\infty,\; {\bf p'=\frac1{p-1}}$. First, we prove the weak inequality
\begin{equation}\label{v4}
\|\widehat f\|_{N_{\bf p',\infty}(M) }\leq c \|f\|_{(L_{p_1,1}(\Bbb R),\ldots,L_{p_n,1}(\Bbb R))}.
\end{equation}
Let    $f\in  L_{\bf p,1}(\Bbb R^n), \; Q\in M,\; Q=Q_1\times\ldots\times Q_n$.  We obtain
$$
\frac1{|Q_1|^{\frac1{p_1}}\ldots|Q_n|^{\frac1{p_n}}}\left|\int_{Q}\widehat f(x)dx\right|
$$
$$
\leq \int_{\Bbb R^n}\left|f(y)\right|\frac1{|Q_1|^{\frac1{p_1}}\ldots|Q_n|^{\frac1{p_n}}}\left|\int_{Q}e^{i(x,y)}dx\right|dy
$$
$$
=\int_{\Bbb R^n}|f(y)|\left|\prod_{j=1}^n\frac{2\sin\frac{|Q_j|y_j}2}{|Q_j|^{\frac1{p_j}}y_j}\right|dy.
$$
Note that
$$
\left|\frac{2\sin\frac{|Q_j|y_j}2}{|Q_j|^{\frac1{p_j}}y_j}\right|\leq \frac{\min (|Q_j|,\frac 2{|y_j|})}{|Q_j|^{\frac1{p_j}}}\leq \left(\frac
2{|y_j|}\right)^{\frac1{p_j'}}.
$$
Hence we have
$$
\frac1{|Q_1|^{\frac1{p_1}}\ldots|Q_n|^{\frac1{p_n}}}\left|\int_{Q}\widehat f(x)dx\right|\leq c \int_{\Bbb R^{n-1}}\prod_{j=2}^{n}\left(\frac
2{|y_j|}\right)^{\frac1{p_j'}}\int_0^\infty t_1^{\frac1{p_1}-1}f^*_1(t_1, y')dtdy'
$$
$$
= c \int_{\Bbb R^{n-1}}\prod_{j=2}^{n}\left(\frac 2{|y_j|}\right)^{\frac1{p_j'}}\|f(\cdot, y')\|_{L_{p_1,1}(\Bbb R)}dy'\leq \ldots
$$
$$
\lesssim \|\ldots \|f(\cdot, y')\|_{L_{p_1,1}(\Bbb R)}\ldots\|_{L_{p_n,1}(\Bbb R)}=\|f\|_{(L_{p_1,1}(\Bbb R),\ldots,L_{p_n,1}(\Bbb R))}.
$$
Taking into account the arbitrariness of $Q\in M$, we obtain
$$
\|\widehat f\|_{N_{\bf p',\infty }(M)}=\sup_{t_j>0}t^{\frac1{p'_1}}_1\ldots t^{\frac1{p'_n}}_n\sup_{{}_{Q\in M}^{|Q_j|\geq
t_j}}\frac1{|Q_1|\ldots|Q_n|}\left|\int_{Q}\widehat f(x)dx\right|
$$

$$
\leq\sup_{Q\in M}\frac1{|Q_1|^{\frac1{p_1}}\ldots|Q_n|^{\frac1{p_n}}}\left|\int_{Q}\widehat f(x)dx\right|\lesssim \|f\|_{(L_{p_1,1}(\Bbb R),\ldots,L_{p_n,1}(\Bbb
R))}.
$$
This proves \eqref{v4}.

Let $1<{\bf p}<\infty$ and ${\bf p_0,p_1}$  satisfy the condition $1<{\bf p_0}<{\bf p}<{\bf p_1}<\infty$. Let $\theta\in (0,1)^n$ be such that
$$
\frac1{\bf p}=\frac{1-\theta}{\bf p_0}+\frac{\theta}{\bf p_1}.
$$

For  $\varepsilon\in E=
\{\varepsilon=(\varepsilon_1,\ldots,\varepsilon_n): \varepsilon_i \in\{0,1\} \}$  we define
${\bf p}_\varepsilon=(p_1^{\varepsilon_1},\ldots,p_n^{\varepsilon_n})$.
Then from \eqref{v4} for the operator $Tf=\widehat f$  it follows
$$
T:{\bf A}_{\bf p_\varepsilon,1}\to N_{\bf p'_\varepsilon,\infty}(M),\quad \varepsilon\in E.
$$
Therefore, according to Theorem \ref{T4-2}, we have
$$
T : L_{\bf p,q} (\mathbb{R}^n) \to N_{\bf p',q}(M).
$$
\end{proof}

\section{Proof of Theorem \ref{T2}}

Let $1<p<r<\infty,  {\bf p}=(p,...,p), {\bf r}=(r,...,r)$.
Inequality \eqref{v3} and \cite[Lemma 1]{Nur00} yields
$$
\|f\|_{L_p}\gtrsim \|\widehat f\|_{N_{\bf p',p}(M)}\gtrsim \|\widehat f\|_{N_{\bf p',r}(M)}
$$
$$
\asymp\left(\sum_{m\in\Bbb Z}\left(2^{\frac{m_1+...+m_n}{p'}}\overline{\widehat f}(2^m;M)\right)^r\right)^{1/r}
=\left(\sum_{k\in\Bbb Z}2^{\frac{kr}{p'}}\sum_{m\in D_k}\left(\overline{\widehat f}(2^m;M)\right)^r\right)^{1/r}
$$
$$
\geq \sup_{k\in\Bbb Z}2^{\frac{k}{p'}}\left(\sum_{m\in D_k}\left(\overline{\widehat f}(2^m;M)\right)^r\right)^{1/r}.
$$
Thus, when $\alpha=\frac1{p'}$ for the quasilinear operator
$$
Tf=\left\{\left(\sum_{m\in D_k}\left(\overline{\widehat f}(2^m;M)\right)^r\right)^{1/r}\right\}_{k\in \Bbb Z},
$$
we have
$$
\|Tf\|_{l^\alpha_\infty}\lesssim\|f\|_{L_p}
$$
for all $p\in(1,r)$.

Now applying interpolation properties of $l^\alpha_s$ and $L_{p,q}$ we arrive at the strong inequality
$$
\|Tf\|_{l^\alpha_q}\lesssim\|f\|_{L_{p,q}},
$$
where $\alpha=\frac1{p'}$, $1<p<r,$ and $0<q\leq\infty.$
It completes the proof of Theorem \ref{T2}.

\end{document}